\documentclass[11pt,a4paper]{amsart}

\usepackage{amssymb,amscd,amsmath,amsthm}
\usepackage{microtype}
\usepackage{graphicx}
\usepackage{mathtools}
\usepackage{hyperref}
\usepackage{thm-restate}
\usepackage{mathrsfs}
\usepackage[british]{babel}

\newtheorem{thm}{Theorem}[section]
\newtheorem{cor}[thm]{Corollary}
\newtheorem{lem}[thm]{Lemma}
\newtheorem{prop}[thm]{Proposition}
\newtheorem*{thm*}{Theorem}
\newtheorem*{prop*}{Proposition}
\theoremstyle{definition}
\newtheorem{defn}[thm]{Definition}
\newtheorem*{defn*}{Definition}
\theoremstyle{remark}
\newtheorem{rem}[thm]{Remark}
\newtheorem{exmp}[thm]{Example}

\newfont{\cusfont}{alnorm}
\newcommand{\alnorm}{\text{\;\cusfont Q\; }}

\DeclareMathOperator{\cd}{cd}
\DeclareMathOperator{\cHom}{cHom}

\DeclareMathOperator{\diam}{diam}

\DeclareMathOperator{\Gcd}{Gcd}
\DeclareMathOperator{\lcm}{lcm}
\DeclareMathOperator{\Hom}{Hom}
\DeclareMathOperator{\Haus}{Haus}
\DeclareMathOperator{\im}{im}

\DeclareMathOperator{\supp}{supp}
\DeclareMathOperator{\tor}{Tor}
\DeclareMathOperator{\vcd}{vcd}

\newcommand{\coarse}{H_\textrm{coarse}}

\newcommand{\bB}{\mathbf{B}}
\newcommand{\bC}{\mathbf{C}}
\newcommand{\bD}{\mathbf{D}}
\newcommand{\bE}{\mathbf{E}}
\newcommand{\bM}{\mathbf{M}}
\newcommand{\bN}{\mathbf{N}}
\newcommand{\bP}{\mathbf{P}}

\newcommand{\bbF}{\mathbb{F}}
\newcommand{\bbN}{\mathbb{N}}
\newcommand{\bbP}{\mathbb{P}}
\newcommand{\bbR}{\mathbb{R}}
\newcommand{\bbZ}{\mathbb{Z}}

\newcommand{\cG}{\mathcal{G}}

\newcommand{\cZ}{\mathcal{Z}}

\title{Groups of cohomological codimension one}
\author{Alexander J. Margolis}
\address{Alexander J. Margolis, Mathematics Department, Technion - Israel Institute of Technology, Haifa, 32000, Israel}
\email{amargolis@campus.technion.ac.il}
\thanks{This research was supported by the Israel Science Foundation (grant No. \texttt{1026/15}).}
\date{23rd August 2019}

\begin{document}

\begin{abstract}
We show that if $H$ is an almost normal subgroup of $G$ such that both $H$ and $G$ are of type $VFP$ and $\vcd(G)=\vcd(H)+1$, then $G$ is the fundamental group of a graph of groups in which all vertex and edge groups are commensurable to $H$. We also investigate almost normal subgroups of one-relator groups and duality groups.
\end{abstract}
\maketitle
\section{Introduction}
In 1968 Stallings proved the following remarkable theorem.
\begin{thm}[\cite{stallings1968torsionfree}]\label{thm:stallingsCD1}
A finitely generated group of cohomological dimension one is free.
\end{thm}
\noindent Swan later proved that all groups of cohomological dimension one  are free \cite{swan1969codim1}.  This was generalised by Dunwoody  to  all groups of cohomological dimension  one over any commutative ring $R$ \cite{dunwoody1979accessibility}.  In this article we prove a higher dimensional generalisation of Theorem \ref{thm:stallingsCD1}.

Let $G$ be a group. Two subgroups $H$ and $K$ are \emph{commensurable} if $H\cap K$ has finite index in both $H$ and $K$. A subgroup $H$ is \emph{almost normal} in $G$, denoted $H\alnorm G$,  if  $H$ is commensurable to $gHg^{-1}$ for every $g\in G$.  Almost normal subgroups are also known as commensurated subgroups, near normal subgroups and inert subgroups.
We say that $G$ is of \emph{type $VFP$} if some finite index subgroup has a  finitely dominated Eilenberg-MacLane space. The \emph{virtual cohomological dimension}  of $G$, denoted $\vcd(G)$, is the cohomological dimension of such a finite index subgroup. The main result of this article is the following:
\begin{thm}\label{thm:mainintro}
Let $H$ and $G$ be groups of type $VFP$ such that $H\alnorm G$ and  $\vcd(G)=\vcd(H)+1$.  Then $G$ is the fundamental group of a finite graph of groups in which every vertex and edge group is commensurable to $H$.
\end{thm} 
In Theorem \ref{thm:mainintro} we do not  assume that $H$ is a codimension one subgroup in the sense of Houghton--Scott, i.e.  that $e(G,H)>1$. In particular, Theorem \ref{thm:mainintro} gives a genuinely new way for showing a group splits. Theorem \ref{thm:mainintro} will be deduced from Theorem \ref{thm:main vfp}. We also prove a similar result for Gorenstein cohomological dimension, a generalisation of cohomological dimension that is finite for a much wider class of groups (see Proposition \ref{prop:CW complex -> finite gcd}).

The hypothesis that $H\alnorm G$ is crucial and cannot be weakened.  Indeed, the Euclidean triangle group $G=\langle a,b,c\mid a^2,b^2,c^2,(ab)^3,(ac)^3,(bc)^3\rangle$ has virtual cohomological dimension two and the infinite cyclic subgroup $H=\langle ab\rangle$ has cohomological dimension one.  However,  $G$ has Serre's property $FA$, so the conclusion of  Theorem \ref{thm:mainintro} does not hold, even though $H$ is (almost) normal in some finite index subgroup of $G$.

Similar results have been obtained by Bieri in the case where $H$ is  normal, although Theorem \ref{thm:mainintro} appears to be  stronger than previously known results even in this case  \cite{bieri1976normal, bieri1978cd2}.
Kropholler proved a  special case of this theorem  when $H$ is infinite cyclic, and then later when $H$ is a Poincar\'e duality group that contains no non-abelian free subgroup \cite{kropholler1990CD2,kropholler2006spectral}. Walker also proved a special of Theorem \ref{thm:mainintro} when $H$ is a  Poincar\'e duality group of non-zero Euler characteristic \cite{walker2010generalisations}. The methods used in the proof of  Theorem \ref{thm:mainintro} are  more geometric than those used by Kropholler and Walker and are of interest in their own right.

Groups containing almost normal subgroups that are not normal often have exotic properties and are a frequent source of counterexamples and pathologies in group theory. For instance:
\begin{enumerate}
\item Baumslag and Solitar constructed the first known examples of  non-Hopfian groups \cite{baumslagsolitar1962}.
\item  Burger and Mozes constructed  finitely presented torsion-free simple groups that act cocompactly on the product of two trees \cite{burgermozes97simple}.
\item Leary and Minasyan constructed CAT(0) groups that are not biautomatic  \cite{leary2019commensurating}.
\end{enumerate}
 These three families of groups all satisfy the hypotheses of Theorem \ref{thm:mainintro}.

This is the second in a sequence of articles examining geometric and topological properties of groups containing almost normal subgroups. A key ingredient in our approach, used extensively in \cite{margolis2019almostnormal}, is that although there is not necessarily a quotient group $G/H$ when $H\alnorm G$, there is still a \emph{quotient space} $G/H$, well-defined up to quasi-isometry, that has many of the properties one might expect from the Cayley graph of the quotient group.  In particular, $G$ can be thought of as a coarse fibre bundle over $G/H$.  
In \cite{margolis2019almostnormal}, the author investigated when this coarse fibre bundle structure is preserved by quasi-isometries. In this article, we use the action of $G$ on the quotient space $G/H$ to prove Theorem \ref{thm:mainintro}.

We briefly outline the proof of Theorem \ref{thm:mainintro}. We first use the hypothesis that $\vcd(G)=\vcd(H)+1$ to show that the quotient space $G/H$ is ``cohomologically one-dimensional''. The main tool that is used is Theorem \ref{thm:coarse kunneth} --- a K\"unneth theorem for coarse bundles. Cohomological one-dimensionality of $G/H$ implies it  has more than one end. We  now mirror Stallings proof of Theorem \ref{thm:stallingsCD1} and show  $G$ splits as an amalgamated free product or HNN extension over a subgroup commensurable to $H$. It follows from a variant of Dunwoody's accessibility theorem  that the process of iteratively splitting $G$ over subgroups commensurable to $H$ eventually terminates \cite{dunwoody1985accessibility}, from which Theorem \ref{thm:mainintro} readily follows.

\subsection*{Applications and other results}

The techniques used in the proof of Theorem \ref{thm:mainintro} show that if $\vcd(H)=\vcd(G)$, then the quotient space is ``cohomologically zero-dimensional'' and hence bounded. We thus deduce the following, which  is stated as part of Lemma \ref{lem:higher cohom vanish}:
\begin{prop}\label{prop:finindexintro}
Suppose $H\alnorm G$ are groups of type $VFP$. Then $H$ is a  finite index subgroup of $G$ if and only if $\vcd(H)=\vcd(G)$.
\end{prop}

We can use Theorem \ref{thm:mainintro} and Proposition \ref{prop:finindexintro} to  investigate properties of  almost normal subgroups in groups of  low cohomological dimension.  The following statement generalises the fact that a finitely generated normal subgroup of a finite rank free group is either trivial or of finite index. 
\begin{cor}\label{cor:alnorm freegp}
If $G$ is a finitely generated free group and $H\alnorm G$ is finitely generated and non-trivial, then $H$ is a finite index subgroup of $G$.
\end{cor}
Corollary \ref{cor:alnorm freegp} can also be deduced from Marshall Hall's theorem.
Things are more interesting in dimension two, where we obtain the following generalisation of \cite[Theorem D]{bieri1978cd2}.
\begin{cor}\label{cor:codim2}
Let $G$ be a finitely presented group of virtual cohomological dimension two. Suppose $H\alnorm G$ is  infinite and finitely presented. Then either:
\begin{enumerate}
\item $H$ is a finite index subgroup of $G$;
\item $H$ is virtually free and $G$ splits as a graph of groups in which all vertex and edge groups are   commensurable to $H.$
\end{enumerate}
\end{cor}
The conclusion of Corollary \ref{cor:codim2} is false if we only assume that $H$ is finitely generated but not finitely presented.  Corollary \ref{cor:codim2}  can be strengthened when $G$ is a one-relator group:

\begin{restatable*}{thm}{onerel}\label{thm:1rel}
Let $G$ be a one-relator group and let $H\alnorm G$ be a  finitely presented almost normal subgroup that is infinite and of infinite index. Then $G$ is torsion-free and two-generated. Moreover, one of the following holds:
\begin{enumerate}
\item  $H$ is infinite cyclic and $G$ splits as a graph of groups in which all vertex and edge groups are infinite cyclic and commensurable to $H$.
\item $G$ contains a free normal subgroup $N$ that is commensurable to $H$ such that $G/N\cong \bbZ$ or $\bbZ_2*\bbZ_2$.
\end{enumerate}
In particular,  $G$ is either a generalised Baumslag--Solitar group or is virtually a free-by-cyclic group.
\end{restatable*}

Bieri showed that if a group $H$ is of type $VFP$ and is a normal subgroup of a duality group $G$, then $H$ and $G/H$  must  also be duality groups  \cite[Theorem A]{bieri1976normal}. We partially generalise this  to the setting of almost normal subgroups. The following  is a special case of Theorems \ref{thm:duality} and \ref{thm:pdn_gps}.
\begin{thm}
Suppose $G$ is a virtual (Poincar\'e) duality group. If $H\alnorm G$ is of type $VFP$, then  $H$ is also a virtual (Poincar\'e) duality group. 
\end{thm}
\subsection*{Organization of paper}
 Sections \ref{sec:prelims} and \ref{sec:ccd} consist of preliminary material. In   Section \ref{sec:prelims} we introduce coarse geometric and topological notions, including the coarse cohomology used in \cite{margolis2018quasi}. In Section \ref{sec:ccd} we discuss various finiteness properties of groups, including Gorenstein cohomological dimension and its relation to coarse cohomology. 
In Section \ref{sec:coarse bundles} we introduce the notion of coarse bundles and investigate their  cohomological properties. In particular, we prove Theorem \ref{thm:coarse kunneth}, a K\"unneth theorem for coarse bundles. We also use  Brown's criterion to deduce Proposition \ref{prop:acyclicity of quotient space}, which demonstrates that quotient spaces are coarsely uniformly acyclic. In Section \ref{sec:main result} we bring all these ingredients together to prove Theorems \ref{thm:main_technical} and \ref{thm:main vfp}, from which Theorem \ref{thm:mainintro} follows easily. In Sections \ref{sec:1rel} and \ref{sec:duality groups} we investigate almost normal subgroups of one-relator groups and duality groups respectively.

\section{Coarse geometry and topology}\label{sec:prelims}
We fix a a PID $R$. All homology and cohomology will be taken with coefficients in $R$.
\subsection{Large-scale geometry}
Given a metric space $(X,d)$, a subset $A\subseteq X$ and  $r\geqslant 0$, the $r$-metric neighbourhood of $A$ is defined to  be \[N_r(A)=\{x\in X\mid d(x,a)\leqslant r \text{ for some $a\in A$}\}.\]  The \emph{Hausdorff distance} $d_{\Haus}(A,B)$ between $A,B\subseteq X$ is defined to be \[\inf \{r\geqslant 0\mid A\subseteq N_r(B) \text{ and } B\subseteq N_r(A)\}.\]
A map $f:X\rightarrow Y$ between topological spaces is said to be \emph{proper} if inverse images of compact sets are compact.

\begin{defn}
A map $f:X\rightarrow Y$ is said to be a \emph{coarse embedding} if there exist proper non-decreasing functions $\eta$, $\phi:\bbR_{\geqslant 0}\rightarrow\bbR_{\geqslant 0}$ such that \[\eta(d_X(x,x'))\leqslant d_Y(f(x),f(x'))\leqslant \phi(d_X(x,x'))\] We say  such an $f$ is an $(\eta,\phi)$-coarse embedding, and we say that $\eta$ and $\phi$ are the \emph{distortion functions}. If there exists an $r\geqslant 0$ such that $N_r(f(X))=Y,$ we say that $f$ is a \emph{coarse equivalence}.  
\end{defn}
\begin{rem}\label{rem:proper inv}
For a proper non-decreasing function $\phi:\mathbb{R}_{\geq 0}\rightarrow \mathbb{R}_{\geq 0}$, there is another proper non-decreasing function $\tilde {\phi}:\mathbb{R}_{\geq 0}\rightarrow \mathbb{R}_{\geq 0}$ by $\tilde{\phi}(R):=\sup(\phi^{-1}([0,R]))$.  This is a sort of inverse to $\phi$ in the following  sense: if $\phi(S)\leq R$, then $S\leq \tilde \phi(R)$, and if $R<\phi(S)$, then $\tilde \phi(R)\leq S$.
\end{rem}

 Being coarsely equivalent defines an equivalence relation on the class of metric spaces. Every discrete countable group admits a proper left-invariant metric, and this metric is unique up to coarse equivalence. 

A \emph{quasi-isometry} is a special kind of coarse equivalence in which the distortion functions can be  taken to be affine.  It is well known that if  $G$ is a finitely generated group, then any two word metrics on $G$ with respect to finite generating sets are quasi-isometric. If $G$ and $H$ are finitely generated groups equipped with the word metric with respect to finite generating sets, then the inclusion $H\rightarrow G$ is a coarse embedding.

Given a metric space $X$ and a parameter $r\geqslant 0$, the \emph{Rips complex} $P_r(X)$ is defined to be the simplicial complex with vertex set $X$, where $\{x_0,\dots, x_n\}$ spans a simplex if $d(x_i,x_j)\leqslant r$ for all $0\leqslant i,j\leqslant n$. 
A metric space  $X$ is said to have \emph{bounded geometry} if for all $r\geqslant 0$, there is an $M_r\geqslant 0$ such that $\lvert N_r(x)\rvert \leqslant M_r$ for every $x\in X$. If $X$ is a bounded geometry metric space, then every Rips complex $P_r(X)$ is locally finite.  We say that a 
metric space is \emph{quasi-geodesic} if it is quasi-isometric to a geodesic metric space. We remark that bounded geometry metric spaces containing more than one point will not be geodesic metric spaces, but may be quasi-geodesic.

\begin{exmp}
The Cayley graph of a finitely generated group is a geodesic metric space but is not a bounded geometry metric space. However, a finitely generated group equipped with the word metric is a quasi-geodesic, bounded geometry metric space.
\end{exmp}

\begin{defn}
We say that a metric space $X$ is \emph{coarsely uniformly $n$-acyclic over $R$} if for every $i\geqslant 0$, $r\geqslant 0$ and $x\in X$, there exist $j=j(i)\geqslant i$ and $s=s(r,i)\geqslant r$ such that the maps \[\widetilde H_k(P_i(N_r(x));R) \rightarrow \widetilde H_k(P_{j}(N_s(x));R),\] induced by inclusion, are zero for $k\leqslant n$. We call $j:\bbR_{\geqslant 0}\rightarrow\bbR_{\geqslant 0}$ and $s:\bbR_{\geqslant 0}\times \bbR_{\geqslant 0}\rightarrow\bbR_{\geqslant 0}$  the \emph{coarse acyclicity functions}. We say that $X$ is \emph{coarsely uniformly acyclic over $R$} if it is coarsely uniformly $n$-acyclic over $R$ for every $n$.
\end{defn}

The notion of coarse uniformly acyclicity is a geometric generalisation of finiteness properties of groups; see Proposition \ref{prop:geomcharoffiniteness}. A crucial step in the proof of Theorem \ref{thm:mainintro} is Proposition \ref{prop:acyclicity of quotient space}, which gives sufficient conditions for the quotient space $G/H$ to be coarsely uniformly acyclic.

\begin{rem}\label{rem:simplifyuniformity}
If $X$ is a bounded geometry metric space admitting a cocompact group action,  coarse uniform $n$-acyclicity is equivalent to the a priori weaker condition that for every $i\geqslant 0$, there exists a $j\geqslant i$  such that the maps \[\widetilde H_k(P_i(X);R) \rightarrow \widetilde H_k(P_{j}(X);R)\] are zero  for $k\leq n$.
\end{rem}

Coarse uniform acyclicity is invariant under coarse equivalences, and hence  invariant under quasi-isometries.
We recall the following characterisation of coarse uniform $0$-acyclicity:
\begin{prop}[{\cite[Proposition 2.18]{margolis2018quasi}}]\label{prop:unif0-acycl}
Let $X$  be a bounded geometry metric space. The following are equivalent:
\begin{enumerate}
\item $X$ is coarsely uniformly $0$-acyclic over $R$;
\item $X$ is coarsely equivalent to a connected locally finite graph equipped with the induced path metric in which edges have length one.
\end{enumerate}
\end{prop}

Suppose  $\Gamma$ is a locally finite graph. If $K$ is a finite subgraph of $\Gamma$, let $c(K)$ denote the number of unbounded components of $\Gamma\setminus K$.  The number of ends of $\Gamma$ is defined to be $\sup\{c(K)\mid \text{$K$ is a finite subgraph}\}\in \bbN\cup \{\infty\}$.
\begin{defn} 
Suppose $X$ is a coarsely uniformly $0$-acyclic, bounded geometry metric space. The \emph{number of ends} of $X$, denoted $e(X)$, is defined to be the number of ends of a locally finite graph that is coarsely equivalent to $X$.
\end{defn}

\subsection{Proper chain complexes}

\begin{defn}
A \emph{based free $R$-module} is a pair $(\bM,\Sigma)$ consisting of a free $R$-module $\bM$ with a distinguished basis $\Sigma$ called the \emph{standard basis}. We define the \emph{support} of $m=\sum_{\sigma\in \Sigma}n_\sigma\sigma$ to be $\supp(m)\coloneqq \{\sigma\in \Sigma\mid n_\sigma\neq 0\}$. 
A \emph{proper map} between based free modules $(\bM,\Sigma)$ and $(\bN,\Lambda)$ is a module homomorphism $f:\bM\rightarrow \bN$ such that  for every $\lambda\in \Lambda$,
\[\#\{\sigma\in \Sigma\mid \lambda\in \supp(f(\sigma))\}<\infty.\]  
A \emph{proper chain complex over $R$} consists of the pair $(\bC_\bullet,\Sigma_\bullet)$, where  $\bC_\bullet$ is a chain complex of $R$-modules such that
\begin{enumerate}
\item $\Sigma_i$ is a basis of $\bC_i$ for every $i$;
\item every boundary map $\partial_i: (\bC_i,\Sigma_i)\rightarrow (\bC_{i-1},\Sigma_{i-1})$ is a proper map between based free modules.
\end{enumerate}  
\end{defn}

To simplify notation,  we frequently say that $f:\bM\rightarrow \bN$ is a proper map or that $\bC_\bullet$ is a proper chain complex, where the choice of standard bases is implicit.
Our motivating example is the cellular chain complex associated to a locally finite CW complex $X$. Each $\bC_i(X)$ is a free module with a standard basis   corresponding to the collection of $i$-cells of $X$. As $X$ is locally finite, each $(i-1)$-cell is contained in the boundary of finitely many $i$-cells, and so the boundary maps $\partial_i: \bC_i(X)\rightarrow \bC_{i-1}(X)$ are proper.

\begin{defn}
A chain map $f_\bullet : \bC_\bullet\rightarrow \bD_\bullet$ between proper chain complexes is said to be \emph{proper} if each map $f_i:\bC_i\rightarrow \bD_i$ is a proper map between based free modules.
 Similarly, a chain homotopy $h_\bullet:\bC_\bullet\rightarrow \bD_{\bullet+1}$ is proper if each map $h_i:\bC_i\rightarrow \bD_{i+1}$ is proper. 
\end{defn}

 Recall that a continuous map between topological spaces is said to be proper if the inverse images of compact sets are compact. If $X$ and $Y$ are locally finite CW complexes and $f:X\rightarrow Y$ is a proper cellular map, then the induced map $f_\bullet:\bC_\bullet(X)\rightarrow \bC_\bullet(Y)$ is a proper chain map. 

If $(\bM,\Sigma)$ is a based free $R$-module, then the \emph{dual module} is $\Hom(\bM,R)$. The \emph{support} of $\alpha\in \Hom(\bM,R)$ is defined to be $\supp(\alpha)=\{\sigma\in \Sigma\mid \alpha(\sigma)\neq 0\}$.
 
\begin{lem}\label{lem:prop map preserves finite support}
Suppose that $(\bM,\Sigma)$ and $(\bN,\Lambda)$ are based free $R$-modules and that $f:\bM\rightarrow \bN$ is proper.  If $\alpha\in \Hom(\bN,R)$ has finite support, then so does $f^*(\alpha)\coloneqq \alpha\circ f\in \Hom(\bM,R)$.
\end{lem}
\begin{proof}
We observe that \[\supp(\alpha\circ f)\subseteq \bigcup_{\lambda\in \supp(\alpha)}\{\sigma\in \Sigma\mid \lambda\in \supp(f(\sigma))\}.\] Since $f$ is proper and $\alpha$ has  finite support,  $\alpha\circ f$ also has finite support.
\end{proof}

Given a proper chain complex $(\bC_\bullet,\Sigma_\bullet)$, we can dualize to obtain the cochain complex $\bC^\bullet=\Hom_R(\bC_\bullet,R)$ with coboundary maps $\delta^\bullet$.  For each $i\in \bbZ$, let $\bC^i_c=\cHom_R(\bC_i,R)\subset \Hom_R(\bC_i,R)$ denote the set of all cochains with finite support. 
\begin{cor}
Let $(\bC_\bullet,\Sigma_\bullet)$ be a proper chain complex. Then $\bC^\bullet_c$ is a sub-cochain complex of $\bC^\bullet$.
\end{cor}
\begin{proof}
If $\alpha\in \bC^i_c$, then $\delta^{i}\alpha\coloneqq \alpha\circ \partial_{i+1}$ has finite support by Lemma \ref{lem:prop map preserves finite support} and properness of  $\partial_{i+1}:\bC_{i+1}\rightarrow \bC_i$.
\end{proof}

Let $H^k_c(\bC_\bullet)$ denote the $k^\text{th}$ cohomology of this cochain complex. We call this the \emph{cohomology with compact supports} of $\bC_\bullet$. When $X$ is a locally finite CW complex, $H^k_c(\bC_\bullet(X))$ is (naturally isomorphic to) the standard  notion of cohomology with compact supports of $X$. We deduce the following by applying Lemma \ref{lem:prop map preserves finite support}.
\begin{cor}
A proper chain map $f_\bullet:\bC_\bullet\rightarrow \bD_\bullet$ between proper chain complexes induces maps $f_c^*:H^*_c(\bD_\bullet)\rightarrow H^*_c(\bC_\bullet)$ in compactly supported cohomology. Moreover, if two proper chain maps $f_\bullet,g_\bullet:\bC_\bullet\rightarrow \bD_\bullet$ are properly chain homotopic, then $f_c^*=g_c^*$.
\end{cor}
 Recall that the tensor product $\bC_\bullet\otimes \bD_\bullet$ of chain complexes is defined to be $(\bC_\bullet\otimes \bD_\bullet)_n=\bigoplus_{i+j=n} \bC_i\otimes \bD_j$ with boundary map \[\partial(\sigma\otimes \lambda)=\partial\sigma\otimes \lambda+ (-1)^i\sigma\otimes \partial\lambda\] when $\sigma\in \bC_i$ and $\lambda\in \bD_j$.  The tensor product of proper chain complexes is also a proper chain complex. Indeed, suppose that $(\bC_\bullet,\Sigma_\bullet)$ and $(\bD_\bullet,\Lambda_\bullet)$ are proper chain complexes. We define the standard basis of  $(\bC_\bullet\otimes \bD_\bullet)_n=\bigoplus_{i+j=n} \bC_i\otimes \bD_j$ to be \[\{\sigma\otimes \lambda\mid \sigma\in \Sigma_i, \lambda\in \Lambda_j, i+j=n\},\] and  $\bC_\bullet\otimes \bD_\bullet$ is a proper chain complex with respect to these bases. 

We now restrict to a special class of proper chain complexes called metric complexes, defined by Kapovich and Kleiner in an appendix of \cite{kapovich2005coarse}. Let $X$ be a bounded geometry metric space. A \emph{free module over $X$} consists of the tuple $(\bM,\Sigma, p)$, where $(\bM,\Sigma)$ is a based free $R$-module  and $p$ is a map $\Sigma\rightarrow X$. We call $X$ the \emph{control space} and say that $(\bM,\Sigma,p)$ has \emph{finite type} if $\sup_{x\in X}\lvert p^{-1}(x)\rvert <\infty$. We define $\supp_X(m)\coloneqq p(\supp(m))$ for every $m\in \bM$.

\begin{defn}
Let $X$ and $X'$ be bounded geometry metric spaces and let $(\bM,\Sigma,p)$ and $(\bN,\Lambda,q)$ be free modules over $X$ and $X'$ respectively. We say that a module homomorphism $\hat f:\bM\rightarrow \bN$ has \emph{finite displacement  over a function $f:X\rightarrow X'$} if there exists an $r\geqslant 0$ such that for every $\sigma\in \Sigma$, \[\supp_{X'}(\hat f(\sigma))\subseteq N_r(f(p(\sigma))).\] If the above holds, we say that \emph{$\hat f$ has displacement at most $r$ over $f$}.
\end{defn}

Maps between free modules over a metric space are often proper:

\begin{prop}
Let $X$ and $X'$ be bounded geometry metric spaces. Suppose that $f:X\rightarrow X'$ is a coarse embedding and that $(\bM,\Sigma,p)$ and $(\bN,\Lambda,q)$ are finite type free modules over $X$ and $X'$ respectively. If $\hat f:\bM\rightarrow \bN$ has finite displacement over $f$, then $\hat f$ is a proper map between based free modules.
\end{prop}
\begin{proof}
Suppose that $\hat f$ has displacement at most $r$ over $f$ and that $f$ is an $(\eta,\phi)$-coarse embedding. We need to show $A_\lambda\coloneqq \{\sigma\in \Sigma\mid \lambda\in \supp(\hat f(\sigma))\}$ is finite for every $\lambda\in \Lambda$. There is nothing to show if $A_\lambda=\emptyset$, so we suppose $\sigma_0\in A_\lambda$.  If $\sigma\in A_\lambda$, then $q(\lambda)\in N_r(f(p(\sigma)))\cap N_r(f(p(\sigma_0)))$, so  $d_{X'}(f(p(\sigma)),f(p(\sigma_0)))\leqslant 2r$. Thus $d_X(p(\sigma),p(\sigma_0))\leqslant \tilde \eta(2r)$, with $\tilde\eta$ as in Remark \ref{rem:proper inv}. There are only finitely many such $\sigma\in \Sigma$, since $X$ has bounded geometry and $\bM$ has finite type.
\end{proof}
Let $X$ be a bounded geometry metric space. A \emph{metric complex over $R$} is a tuple $(X,\bC_\bullet, \Sigma_\bullet, p_\bullet)$ such that:
\begin{enumerate}
\item $\bC_\bullet$ is a non-negative chain complex of $R$-modules, i.e. $\bC_i=0$ for $i<0$;
\item each $(\bC_i,\Sigma_i,p_i)$ is a finite type free $R$-module over $X$;
\item each boundary map $\partial_i:\bC_i\rightarrow \bC_{i-1}$ has finite displacement over the identity map on $X$;
\item the composition $\varepsilon\circ \partial_1$ is zero, where $\varepsilon:C_0\rightarrow R$ is the augmentation map given by $\sigma\mapsto 1_R$ for each $\sigma\in \Sigma_0$;
\item the map $p_0:\Sigma_0\rightarrow X$ is onto.
\end{enumerate} 
 We refer the reader to \cite{kapovich2005coarse} and \cite{margolis2018quasi} for more details on the theory of metric complexes. When unambiguous, we denote the metric complex $(X,\bC_\bullet, \Sigma_\bullet, p_\bullet)$  simply by $(X,\bC_\bullet)$, or even by $\bC_\bullet$. The metric space $X$ is said to be  \emph{control space} of $(X,\bC_\bullet)$. We say that $(X,\bC_\bullet)$ is \emph{$n$-dimensional} if $\bC_i=0$ for $i>n$. 

We say a metric complex $(A,\bC'_\bullet,\Sigma'_\bullet, p'_\bullet)$ is a  \emph{subcomplex}  of  $(X,\bC_\bullet, \Sigma_\bullet, p_\bullet)$ if $A\subseteq X$,  $\bC'_\bullet$ is  a subchain complex of $\bC_\bullet$ and for each $i$,  $\Sigma'_i\subseteq \Sigma_i$ and $p'_i=p_i|_{\Sigma'_i}$. 
Given a subset $A\subseteq X$, we define $(A,\bC_\bullet[A])$ to be the largest subcomplex of $(X,\bC_\bullet)$ with control space $A$.  The \emph{$n$-skeleton} of $(A,\bC_\bullet[A])$ is the metric  complex $(A,\bC_\bullet[A]_n)$, where $\bC_i[A]_n=\bC_i[A]$ when $i\leq n$ and $\bC_i[A]_n=0$ when $i>n$.

\begin{defn}
We say $(X,\bC_\bullet)$ is \emph{uniformly $n$-acyclic} if for all $r$, there exists a $\mu(r)\geqslant r$ such that for all $x\in X$ and $k\leqslant n$, the map \[\widetilde H_k(\bC_\bullet[N_r(x)])\rightarrow \widetilde H_k(\bC_\bullet[N_{\mu(r)}(x)]),\] induced by inclusion, is zero. We say that $\mu:\bbR_{\geqslant 0}\rightarrow \bbR_{\geqslant 0}$ is the \emph{acyclicity function}.
\end{defn}

\begin{prop}[{\cite{kapovich2005coarse},\cite[Proposition 3.20]{margolis2018quasi}}]
A bounded geometry metric space $X$ is coarsely uniformly $(n-1)$-acyclic over $R$  if and only if it is the control space of an $n$-dimensional uniformly $(n-1)$-acyclic metric complex $(X,\bC_\bullet)$ over $R$. Moreover, the acyclicity function  and the displacement of the boundary maps of $(X,\bC_\bullet)$ can be bounded as a function of the coarse acyclicity functions of $X$.
\end{prop}

We recall the notion of \emph{coarse cohomology} used in \cite{margolis2018quasi} (see also \cite{roe1993coarse} and \cite{kapovich2005coarse}).
\begin{defn}
Fix $n>0$. Suppose  $X$ is a bounded geometry metric space that is coarsely uniformly $(n-1)$-acyclic over $R$. Let $(X,\bC_\bullet)$ be an $n$-dimensional, uniformly $(n-1)$-acyclic metric complex over $R$. For $k<n$, we define $\coarse^k(X;R)\coloneqq H^k_c(\bC_\bullet)$.
\end{defn}
When $X$ is coarsely uniformly $(n-1)$-acyclic, it is possible to define  $\coarse ^n(X;R)$ using  an idea originally due to Geoghegan--Mihalik \cite{geoghegan1986note}. This will not be needed here, but we refer the reader to \cite{margolis2018quasi} for more details. The preceding definition of coarse cohomology may differ slightly from the coarse cohomology defined in \cite{margolis2018quasi} in dimension zero. The former should be thought of as unreduced cohomology, while the latter is reduced (compare Proposition \ref{prop:cohomatinf} below to Remark 3.24 in \cite{margolis2018quasi}).

In \cite[Proposition 3.26]{margolis2018quasi}, it is shown that coarse equivalences preserve coarse cohomology. In particular, $\coarse^k(X;R)$ is independent of the  choice of metric complex $(X,\bC_\bullet)$. Coarse cohomology contains information about the topology at infinity of the space $X$ (see \cite[\S 3.5]{margolis2018quasi}). In particular, coarse cohomology in dimensions 0 and 1 can be easily characterised: 
\begin{prop}\label{prop:cohomatinf}
If $X$ is a bounded geometry, coarsely uniformly $0$-acyclic metric space, then
\begin{enumerate}
\item\label{item:cohomatinf0}
$\coarse^0(X;R)\cong \begin{cases}
	R & \text{if $X$ is bounded}\\
	0 & \text{otherwise.}
\end{cases}$ 
\item If $X$ is unbounded and coarsely uniformly $1$-acyclic over $R$, then $\coarse^1(X;R)$ is a free $R$-module with  $\text{rank} (\coarse^1(X;R))=e(X)-1$. \label{item:cohomatinf1}
\end{enumerate}
\end{prop}
The condition that $X$ is  coarsely uniformly $1$-acyclic over $R$ can be removed, although we do not need this more general formulation of (\ref{item:cohomatinf1}) in this article.
\begin{proof}
Let $(X,\bC_\bullet,\Sigma_\bullet,p_\bullet)$ be a uniformly $0$-acyclic metric complex over $R$. Suppose $\sigma_0\in \Sigma$ and $\alpha\in \bC^0_c$ is a compactly supported cocycle. For every  $\sigma\in \Sigma_0$, we have that $\alpha(\sigma)=\alpha(\sigma_0)$ since $\sigma-\sigma_0$ is a boundary and $\alpha$ is a cocycle.  Thus $\alpha\in\bC^0_c$ is determined by its value on $\sigma_0$. 
When $X$ is unbounded, since $\supp(\alpha)$ was assumed to be finite, we deduce that $\alpha\equiv 0$. When $X$ is bounded, the map $H^0_c(C_\bullet)\rightarrow R$ given by $[\alpha]\mapsto \alpha(\sigma_0)$ is easily seen to be an isomorphism.
The proof of (\ref{item:cohomatinf1}) is standard and follows from the arguments in \cite[\S 13.4]{geoghegan2008topological} and \cite[\S 3.5]{margolis2018quasi}.
\end{proof}

We conclude with a universal coefficient theorem for coarse cohomology. Let $R$ and $S$ be PIDs and suppose $\iota: R\rightarrow S$ is  a ring homomorphism. Then $S$ can be thought of as an $R$-module via $r\cdot s=\iota(r)s$.  We now have the following universal coefficient theorem for coarse cohomology:
\begin{prop}\label{prop:univ coefficient}
Let $R$ and $S$ be as above. Suppose $X$ is coarsely uniformly $n$-acyclic over $R$. Then $X$ is coarsely uniformly $n$-acyclic over $S$ and for each $k<n$, we have a split short exact sequence
\[
0\rightarrow\coarse^k(X;R)\otimes_R S \rightarrow \coarse^k(X;S)\rightarrow \tor^R_1(\coarse^{k+1}(X;R), S)\rightarrow 1.
\]
\end{prop}
\begin{proof}
Let  $(X,\bC_\bullet,\Sigma_\bullet,p_\bullet)$ be a uniformly $n$-acyclic metric complex over $R$. Then $\bC_\bullet\otimes_R S$ is a uniformly $n$-acyclic metric complex over $S$. Thus 
	\[\coarse^k(X;S)\cong H^k(\cHom_S(\bC_\bullet\otimes_R S,S)),\] where  $\cHom_S(\bC_\bullet\otimes_R S,S)$ consists of $S$-module  homomorphisms from $\bC_\bullet\otimes_R S$ to $S$ with finite supports.

We now define a map
\[\phi: \cHom_R(\bC_\bullet,R)\otimes_R S\rightarrow \cHom_S(\bC_\bullet\otimes_R S,S)\] given by $\phi(\alpha\otimes s)(\rho\otimes s')=\alpha(\rho)\cdot ss'$ for $\alpha \in \cHom_R(\bC_\bullet,R)$, $\rho\in \bC_\bullet$ and $s,s'\in S$. It is straightforward to verify that $\phi$ is an isomorphism of cochain complexes. The key point is the standard basis $\Sigma_i$ of $\bC_i$ induces  dual bases of $\cHom_R(\bC_i,R)\otimes_R S$ and $\cHom_S(\bC_i\otimes_R S,S)$. This is not the case if homomorphisms are allowed to have infinite support.  
We now apply the standard K\"unneth formula for tensor products of chain complexes (see Proposition \ref{prop:kunnethchain}) to obtain the desired short exact sequence.
\end{proof}

\section{Cohomological dimension and group cohomology}\label{sec:ccd}
Let $R$ be a PID. We say that a group $G$ is of \emph{type $FP_n(R)$}  if the trivial $RG$-module $R$ has a projective resolution $\bP_\bullet\rightarrow R$  such that $\bP_i$ is finitely generated as an $RG$-module for $i\leq n$. We say that a group $G$  is of type $FP_\infty(R)$ if it is of type $FP_n(R)$ for every $n$.
A group is finitely generated if and only if it is of type $FP_1(\bbZ)$. We say that a group is \emph{almost finitely presented} if it is of type $FP_2(\bbZ_2)$. Every finitely presented group is almost finitely presented. 
These properties can be characterised geometrically:
\begin{prop}[See \cite{kapovich2005coarse}, \cite{drutu2018geometric} and \cite{margolis2018quasi}] \label{prop:geomcharoffiniteness}
Suppose that $G$ is a discrete countable group equipped with a left-invariant proper metric. Then:
\begin{enumerate}
\item $G$ is of type $FP_n(R)$ if and only if $G$ is coarsely uniformly $(n-1)$-acyclic over $R$;
\item $G$ is of type $FP_\infty(R)$ if and only if $G$ is coarsely uniformly acyclic over $R$.
\end{enumerate}
\end{prop}

The \emph{cohomological dimension} of a group $G$, denoted  $\cd_R(G)$, is defined to be the least $n$ such that the trivial $RG$-module $R$ has a projective resolution of length $n$, i.e. a projective resolution $\bP_\bullet\rightarrow R$  with $\bP_i=0$ for $i>n$. A group is of \emph{type $FP(R)$} if there exists a finite length resolution of $R$ by finitely generated projective $RG$-modules.  A group is said to be of \emph{type $VFP(R)$} if it has a finite index subgroup of type $FP(R)$. If $G$ is of type $VFP(R)$, then the \emph{virtual cohomological dimension of $G$}, $\vcd_R(G)$, is the cohomological dimension of a finite index subgroup of type $FP(R)$.

When the  ring $R$ is omitted from notation, it will be assumed that $R=\bbZ$. In other words, $\cd(G)$ denotes $\cd_\bbZ(G)$, type $VFP$ denotes type $VFP(\bbZ)$ etc. If $G$ is of type $FP_n$, then it is of type $FP_n(R)$ for any $R$. 
We frequently make use of the following fact:
\begin{prop}[\cite{brown1982cohomology}]\label{prop:VFP max cohom nonvanish}
If $G$ is of type $VFP(R)$, then \[\vcd_R(G)=\max\{n\mid H^n(G,RG)\neq 0\}.\]
\end{prop}

The cohomological  dimension of a group over $R$ sometimes  depends on the choice of $R$. For instance, Dicks and Leary have shown that there exists a group $G$  such that $\cd_\bbF(G)< \cd(G)<\infty$ for any field $\bbF$ \cite{dicksleary1998subgroup}. This cannot happen for groups of type $VFP(R)$, as Proposition \ref{prop:type vf} demonstrates.

If $R$ is a PID,  let $\bbP_R$ denote the collection of primes in $R$. For each $p\in \bbP_R$, let $R_p$ denote the field $R/pR$ and let $R_0$ denote the field of fractions of $R$. We use the following result, which follows easily from \cite[Lemma 3.6]{bieri1976normal} and the universal coefficient theorem for group cohomology.
\begin{prop}\label{prop:type vf}
Let $R$ be a PID and suppose $G$ is of type $VFP(R)$.  Then $\vcd_{R_i}(G)\leq\vcd_R(G)$ for every $i\in \bbP_R\cup\{0\}$, and   $\vcd_{R_i}(G)=\vcd_R(G)$ for some $i\in \bbP_R\cup\{0\}$.
\end{prop}
Proposition \ref{prop:type vf} plays a crucial role in the proof of Theorem \ref{thm:main vfp}, allowing us to reduce the general case  to the case where $R$ is a field.
The following was noted as Proposition 3.28 in \cite{margolis2018quasi}, and follows easily from \cite[\S VIII, Proposition 7.5]{brown1982cohomology}; see also  \cite{gersten1993quasi}. 
\begin{prop}\label{prop:coarse cohom vs cohom}
If $G$ is a group of type $FP_n(R)$, then $H^k(G,RG)$ and $\coarse^k(G;R)$ are isomorphic as $R$-modules for any $k< n$. In particular, $H^k(G,RG)$ is a quasi-isometry invariant amongst groups of type $FP_\infty(R)$.
\end{prop}

We now introduce the notion of  Gorenstein cohomological dimension of groups, a generalisation of (virtual) cohomological dimension that allows for proper group actions rather than free group actions. We refer the reader to \cite{holm2004} for definitions and properties of Gorenstein projective modules. We say that a group $G$ has finite \emph{Gorenstein cohomological dimension} over $R$, denoted $\Gcd_R(G)$, if the trivial $RG$-module $R$ admits a finite length resolution by Gorenstein projective modules.

As the following proposition shows, the class of groups with finite Gorenstein cohomological dimension is much broader than the class of groups with finite virtual cohomological dimension. In particular, such groups are not required to be virtually torsion-free.
\begin{prop}[{\cite[Proposition 3.1]{bahlekeh2009gorenstein}} and {\cite[Proposition 2.1]{emmanouil2018gorenstein}}]\label{prop:CW complex -> finite gcd}
Suppose a group $G$ acts properly and cellularly on a contractible finite-dimensional  CW complex $X$. Then $\Gcd_R(G)\leq \dim(X)<\infty$ for any commutative ring $R$.
\end{prop}

We now state some properties of Gorentstein cohomological dimension.
\begin{prop}[\cite{emmanouil2018gorenstein}]\label{prop:gorenstein}
Suppose  $G$ is a group of type $FP_\infty(R)$ with $\Gcd_R(G)<\infty$, where $R$ is a PID. Then:
\begin{enumerate}
\item $\Gcd_R(G)=\max\{n\mid H^n(G,RG)\neq0\}$;
\item if $H\leq G$, then $\Gcd_R(H)\leq \Gcd_R(G)$.
\end{enumerate}
In particular, if $\vcd_R(G)<\infty$, then $\Gcd_R(G)=\vcd_R(G)$.
\end{prop}
As we observe in Remark \ref{rem:ccdvsgcd},  these are the only properties of Gorenstein cohomological dimension that we will make use in the proof of Theorem \ref{thm:main_technical}. In particular, no knowledge of Gorenstein homological algebra will be needed in what follows.

We deduce the following from Propositions \ref{prop:coarse cohom vs cohom} and \ref{prop:gorenstein}:
\begin{cor}
If $G$ and $G'$ are quasi-isometric groups of type $FP_\infty(R)$ with $\Gcd_R(G),\Gcd_R(G')<\infty$, then $\Gcd_R(G)=\Gcd_R(G')$.
\end{cor}
It is likely that the $FP_\infty(R)$ condition can be removed using the techniques of Sauer \cite{sauer2006cdqi}.
We now state some more properties of Gorenstein cohomological dimension that might be of further interest to geometric group theorists.
\begin{prop}
If $G$ is either a Gromov hyperbolic or a CAT(0) group,  then $\Gcd_R(G)<\infty$ for any PID $R$.
\end{prop}
\begin{proof}
If $G$ is hyperbolic,  then the Rips complex $P_r(G)$ is contractible for $r$ sufficiently large. If $G$ is CAT(0), it follows from Remark III.$\Gamma$.3.27 of \cite{bridson1999metric} that $G$ acts properly and cocompactly on some finite-dimensional simplicial complex.  In both cases, Proposition \ref{prop:CW complex -> finite gcd} ensures that $\gcd_R(G)<\infty$.
\end{proof}

The $\cZ$-boundary was first defined by Bestvina and later generalised by Dranishnikov \cite{bestvina1996zboundary, dranishnikov2006bestvinamess}. The $\cZ$-boundary of a hyperbolic group is simply its Gromov boundary, and the $\cZ$-boundary of a CAT(0) group $G$ is a visual boundary $\partial X$ of a CAT(0) space $X$ admitting a proper cocompact $G$-action. 
Gorenstein cohomological dimension has a geometric interpretation for groups admitting a $\cZ$-boundary in the sense of \cite{dranishnikov2006bestvinamess}. If $Z$ is a topological space, let $\dim(Z)$ be the Lebesgue covering dimension of $X$ and let $\dim_R(Z)$ be the cohomological dimension of $Z$ with coefficients in $R$.
\begin{prop}
Suppose that $G$ is a group of finite Gorenstein cohomological dimension that admits  $\cZ$-boundary $Z$ in the sense of \cite{dranishnikov2006bestvinamess}. Then $\dim_R(Z)+1= \Gcd_R(G)$ and $\dim(Z)+1=\Gcd(G)$.
\end{prop}
\begin{proof}
Suppose $(\overline{X},Z)$ is a $\cZ$-structure on $G$ in the sense of \cite{dranishnikov2006bestvinamess}. Since $\overline{X}$ is contractible,  the long exact sequence in cohomology gives $H^{n+1}_c(X;R)\cong H^{n}(Z;R)$ (see \cite[Proposition 1.5]{bestvina1996zboundary}). As $G$ is quasi-isometric to $X$, $H^{n+1}(G,RG)\cong\coarse^{n+1}(G;R)\cong H^{n}(Z;R)$ as $R$-modules. (This argument is used in the proof of \cite[Corollary 2]{dranishnikov2006bestvinamess}, noting $\coarse^{n+1}(G;R)$ is isomorphic to Roe's coarse cohomology as mentioned in \cite[Appendix B]{margolis2018quasi}.)

It now follows from the main result of \cite{dranishnikov2006bestvinamess} and Proposition \ref{prop:gorenstein} that $\dim_R(Z)+1= \Gcd_R(G)$ for any PID $R$. It was shown in \cite{moran2016finite} that $Z$ is finite dimensional, so $\dim(Z)=\dim_\bbZ(Z)$.  Thus $\dim(Z)+1= \Gcd(G)$.
\end{proof}

\section{Cohomology of coarse bundles}\label{sec:coarse bundles}
There have been several adaptations of the theory of fibre bundles and fibrations to the setting of metric spaces, see \cite{farbmosher2000abelianbycyclic}, \cite{kapovich2005coarse}, \cite{whyte2010coarse} and \cite{mj12bundles} for more information.
The following definition is essentially the same as that used in \cite{whyte2010coarse}:
\begin{defn}[\cite{margolis2019almostnormal}]
Let $X$, $F$ and $B$ be bounded geometry, quasi-geodesic metric spaces. We say that $p:X\rightarrow B$ is a \emph{coarse bundle} with fibre $F$ if there exist constants $K\geqslant 1, A,E\geqslant 0$ such that the following hold:
\begin{enumerate}
\item $d(p(x),p(x'))\leqslant K d(x,x')+A$ for all $x,x'\in X$.
\item Let $D_b\coloneqq p^{-1}(N_E(b))$ for all $b\in B$. Then there is an $(\eta,\phi)$-coarse embedding $s_b:F\rightarrow X$ such that $d_{\Haus}(D_b, \im({s_b}))\leqslant A$, where $\eta$ and $\phi$ can be chosen independently of $b$.
\item $d_{\Haus}(D_b,D_{b'})\leqslant Kd(b,b')+A$ for all $b,b'\in B$.
\end{enumerate}
We say that each $D_b$ is a \emph{fibre} of $X$.
\end{defn}

Suppose  $G$ is a finitely generated group equipped with the word metric with respect to a finite generating set, and let $H\alnorm G$. We define a metric on $G/H$ by $d(gH,kH)\coloneqq d_{\Haus}(gH,kH)$ for all $gH,kH\in G/H$. The resulting metric space $G/H$ is called the \emph{quotient space}. The following is shown in \cite{margolis2019almostnormal}; see also \cite{kronmoller08roughcayley}, \cite{whyte2010coarse} and \cite{connermihalik2014}.
\begin{prop}
Let $G$ and $H$ be as above. Then the quotient space $G/H$ is well-defined up to quasi-isometry, and is a bounded geometry, quasi-geodesic metric space. Moreover, the  quotient map $p:G\rightarrow G/H$ given by $g\mapsto gH$ is a coarse bundle with fibre $H$.
\end{prop}

In order to apply coarse topological methods to the quotient space, we need to show it is coarsely uniformly acyclic. Fortunately, this can easily be done by applying Brown's criterion.

An \emph{$n$-good $G$-CW complex over $R$} is a CW complex $X$ admitting a cellular $G$-action such that:
\begin{enumerate}
\item $\widetilde H_k(X;R)=0$ for $k<n$;
\item for $0\leq p\leq n$, the stabilizer of any $p$-cell of $X$ is of type $FP_{n-p}(R)$.
\end{enumerate}
A \emph{filtration}  of $X$ is a nested sequence $X_1\subseteq X_2 \subseteq \dots $ of $G$-invariant subcomplexes of $X$ such that $X=\cup_{i\in \bbN} X_i$. We say a filtration $(X_i)$ has \emph{finite $n$-type} if every $X_i$ has finitely many $G$-orbits of $p$-cells for $p\leqslant n$. Moreover, a filtration is said to be \emph{essentially $(n-1)$-acyclic over $R$} if for every $k<n$ and $i$, there is a $j\geq i$ such that the map $\widetilde H_k(X_i;R)\rightarrow \widetilde H_k(X_j;R)$, induced by inclusion, is trivial. The following is a simplified version of Brown's criterion.

\begin{thm}[Brown's Criterion, \cite{brown1987finiteness}]\label{thm:browncrit}
Suppose $X$ is an $n$-good $G$-CW complex over $R$ admitting a finite $n$-type filtration $(X_i)$. This filtration is essentially $(n-1)$-acyclic over $R$ if and only if $G$ is of type $FP_n(R)$.
\end{thm}

\begin{prop}\label{prop:acyclicity of quotient space}
Let $G$ be a group of type $FP_n(R)$ containing an almost normal subgroup $H$ of type $FP_n(R)$. Then the quotient space $G/H$ is coarsely uniformly $(n-1)$-acyclic over $R$.
\end{prop}
\begin{proof}
Consider the infinite-dimensional simplex $P_\infty(G/H)$, whose vertices are all the left $H$-cosets. The stabilizer of each finite face $\{g_0H,\dots, g_mH\}$ is commensurable to $H$, so is of type $FP_n(R)$. As $P_\infty(G/H)$ is contractible,  $P_\infty(G/H)$ is an $n$-good $G$-CW complex. The filtration $(P_i(G/H))$ is of finite $n$-type, since $G/H$ has bounded geometry. As $G$ is of type $FP_n(R)$, Theorem \ref{thm:browncrit} ensures  $(P_i(G/H))$ is essentially $(n-1)$-acyclic over $R$. Since $G/H$ is proper and admits a cobounded $G$-action, Remark \ref{rem:simplifyuniformity} implies that $G/H$ is coarsely uniformly $(n-1)$-acyclic over $R$.
\end{proof}

The remainder of this section is devoted to a proof of the following:

\begin{thm}[A K\"unneth theorem for coarse bundles]\label{thm:coarse kunneth}
Let $R$ be a PID. Suppose $p:X\rightarrow B$ is a fibre bundle with fibre $F$ such that $F$ and $B$ are  coarsely uniformly acyclic over $R$. Then $X$ is coarsely uniformly acyclic over $R$ and for every $k\in \bbN$ there is a split short exact sequence of $R$-modules 
\begin{align*}
0&\rightarrow\bigoplus_{i+j=k}\coarse^i(F;R)\otimes_R \coarse^j(B;R)\rightarrow \coarse^k(X;R)\\&\rightarrow \bigoplus_{i+j=k+1}\tor^R_1(\coarse^i(F;R), \coarse^j(B;R))\rightarrow 1.
\end{align*}
\end{thm}

All the ideas needed to prove Theorem \ref{thm:coarse kunneth} are found in \cite[\S 11.5]{kapovich2005coarse}. Indeed, although Theorem \ref{thm:coarse kunneth} is not explicitly stated there, a slightly less general version of Theorem \ref{thm:coarse kunneth} is implicitly used in \cite{kapovich2005coarse}. Group theoretic analogues of  Theorem \ref{thm:coarse kunneth} when $H\vartriangleleft G$ are well-known, for instance see \cite[\S 17.3]{geoghegan2008topological}.

Let $(B,\bB_\bullet,\Sigma_\bullet,p_\bullet)$ be a uniformly acyclic metric complex with control space $B$.
For each $b\in B$, we fix an uniformly acyclic  metric complex $(D_b,\bD^b_\bullet,\Sigma^b_\bullet,p_\bullet^b)$  with control space $D_b$. The displacement and acyclicity constants associated to $\bD_\bullet^b$  can be chosen independently of $b$. Moreover, it can also be assumed that the metric complexes $(D_b,\bD^b_\bullet)$ have \emph{uniform finite type} independent of $b$, i.e.  $\sup_{b\in B, x\in D_b}\lvert(p_i^b)^{-1}(x)\rvert<\infty$ for every $i$. 

  For each $b,b'\in B$, let $f^{b,b'}$ be a closest point projection $D_b\rightarrow D_{b'}$, and let $f^{b,b'}_\#:\bD^b_\bullet\rightarrow \bD^{b'}_\bullet$ be a chain map with  finite displacement  over $f^{b,b'},$ where the displacement depends only on $d_B(b,b')$. We may assume  that when $b=b'$, $f^{b,b'}$ and $f_\#^{b,b'}$ are the identity maps.

Every $k$-chain in $\tau \in \bB_\bullet\bigotimes \oplus_{b\in B}\bD^b_\bullet$ can be written uniquely in the form $\tau=\Sigma_r n_r(\sigma_r\otimes \lambda_r)$, where $\sigma_r\in \Sigma_{i_r}$ and $\lambda_r\in \Sigma_{k-i_r}^{b_r}$ for some $b_r\in B$. We define 
\begin{equation}\label{eqn:support of chains}
\supp_X(\tau)\coloneqq \bigcup\{p_{k-i_r}^{b_r}(\lambda_r)\mid n_r\neq 0\}\subseteq X
\end{equation}
and 
\begin{equation}\label{eqn:B-support of chains}
\supp_B(\tau)\coloneqq \bigcup\{p_{i_r}(\sigma_r)\mid n_r\neq 0\}\subseteq B.
\end{equation}
 We define $\bE_{i,j}$ to be the  free module with basis \[T_{i,j}\coloneqq \{\sigma \otimes \lambda \mid  \sigma\in \Sigma_i, \lambda\in \Sigma^{b}_j, p_i(\sigma)=b\}, \] which we identify with the submodule of $\bB_i\bigotimes \oplus_{b\in B}\bD^b_j$ generated by $T_{i,j}$. We set $T_k\coloneqq \bigsqcup_{i+j=k}T_{i,j}$ and $\bE_k=\bigoplus_{i+j=k}\bE_{i,j}$ and define $q_k:T_k\rightarrow X$ by $q_k(\sigma\otimes \lambda)=p_{k-i}^{p_i(\sigma)}(\lambda)$ for every $\sigma\otimes \lambda\in T_{i,k-i}$. Since the metric complexes $(D_b,\bD^b_\bullet)$ have uniform finite type, each $(\bE_k,T_k,q_k)$ is a finite type free module over $X$.

It is not necessarily the case that $\bE_\bullet$ is a subcomplex of $\bB_\bullet\bigotimes \oplus_{b\in B}\bD^b_\bullet$. However, we can define boundary maps so that $\bE_\bullet$ is a chain complex.
\begin{lem}\label{lem:defchain complex}
For every $k\in \bbZ$, there exists a  boundary map $\partial:\bE_k\rightarrow \bE_{k-1}$ such that \[(X,\bE_\bullet,T_\bullet, q_\bullet)\] is a uniformly acyclic metric complex over $R$. Moreover, $\bE_\bullet$ is properly chain homotopic to $\bB_\bullet\otimes \bD_\bullet^b$ for any $b\in B$.
\end{lem}
Before proving Lemma \ref{lem:defchain complex}, we explain how to deduce Theorem \ref{thm:coarse kunneth} from it. To do this, we require  the ordinary K\"unneth theorem for  chain complexes of $R$-modules, where $R$ is assumed to be a PID:
\begin{prop}[{\cite[Proposition I.0.8]{brown1982cohomology}}]\label{prop:kunnethchain}
Let $\bC_\bullet$ and $\bD_\bullet$ be chain complexes of free $R$-modules. Then there is a natural short exact sequence
\begin{align*}
0&\rightarrow\bigoplus_{i+j=k}H_i(\bC_\bullet)\otimes H_j(\bD_\bullet)\rightarrow H_k(\bC_\bullet\otimes \bD_\bullet)\\&\rightarrow \bigoplus_{i+j=k-1}\tor^R_1(H_i(\bC_\bullet), H_j(\bD_\bullet))\rightarrow 1
\end{align*}
that splits.
\end{prop}

\begin{proof}[Proof of Theorem \ref{thm:coarse kunneth}]
As $(X,\bE_\bullet)$ is a uniformly acyclic metric complex over $R$, it follows from the definition of coarse cohomology that $\coarse^k(X;R)=H^k_c(\bE_\bullet)$. Since $\bE_\bullet$ is properly chain homotopic to $\bB_\bullet\otimes \bD^b_\bullet$ for some fixed $b\in B$, we see that $\coarse^k(X;R)\cong H^k_c(\bB_\bullet\otimes \bD^b_\bullet)$. We conclude by applying  Proposition \ref{prop:kunnethchain}, which can be done because the cohomology of a cochain complex is simply the homology of the associated chain complex with indices reversed.
\end{proof}
We now prove Lemma \ref{lem:defchain complex}.
  To define the boundary map on $\bE_\bullet$, we define a  chain map
 \[g^b_\#: \bB_\bullet\otimes \bD^b_\bullet\rightarrow \bE_\bullet  \] for every $b\in B$.    The boundary map on $\bE_\bullet$ will then be defined by 
 \[\partial(\sigma\otimes \lambda)=g^{p_i(\sigma)}_\#(\partial\sigma \otimes \lambda)+(-1)^i\sigma \otimes \partial\lambda\]
  for every $\sigma\otimes \lambda\in T_{i,j}$. This might seem rather circular, as the boundary map on $\bE_\bullet$  is defined in terms of  $g_\#^b$, whilst $g_\#^b$ is defined to be a chain map, which implies boundary maps of $\bE_\bullet$ have already been chosen! However, we will inductively define  $\partial$ and $g_\#^b$  on filtrations of $\bE_\bullet$ and $\bB_\bullet\otimes \bD^b_\bullet$ so this makes sense.

Whilst defining each $g_\#^b$,  we show that there is a constant $M=M_{i,j}\geqslant 0$ and a function  $\phi=\phi_{i,j}:\bbR_{\geqslant 0}\rightarrow \bbR_{\geqslant 0}$, both independent of $b$, such that  
\begin{align}\label{eqn:findisp of g}
\supp_X(g_\#^b(\sigma\otimes \lambda))&\subseteq N_{\phi(d_B(p_i(\sigma),b))}(p^b_j(\lambda))\\
\supp_B(g_\#^b(\sigma\otimes \lambda))&\subseteq N_{M}(p_i(\sigma))\label{eqn:finBdisp of g}
\end{align} for all  $\sigma\in \Sigma_i$ and $\lambda\in \Sigma^b_j$.

We proceed inductively,  defining a nested sequence of chain complexes $\bE^0_\bullet\leqslant \bE^1_\bullet\leqslant \dots $ with $\bE^k_\bullet=\oplus_{i\leqslant k}\bE_{i,\bullet-i}$. We show that
for each $\sigma\in \Sigma_i$ and $ \lambda\in \Sigma^b_j$, we have
\begin{equation}
\label{eqn:cycle_g}
g^b_\#(\sigma\otimes \lambda)-\sigma\otimes f_\#^{b,p_i(\sigma)}(\lambda)\in \bE^{i-1}_\bullet. 
\end{equation}

For the base case, we define the boundary map on $\bE^0_i=\bE_{0,i}$  by $\partial(\sigma\otimes \lambda)=\sigma\otimes \partial\lambda$.
For each $b\in B$, we define a  chain map  \[g^b_\#: [\bB_\bullet]_0\otimes \bD^b_\bullet\rightarrow \bE^0_\bullet  \] given by  $g_b(\sigma\otimes \gamma)=\sigma \otimes f^{b,p_0(\sigma)}_\#(\gamma)$. Conditions (\ref{eqn:findisp of g}), (\ref{eqn:finBdisp of g}) and  (\ref{eqn:cycle_g}) are automatically satisfied.

  We now assume that the chain map $g^b_\#$ has already been defined on  $[\bB_\bullet]_{k-1}\otimes \bD^b_\bullet$ and that  (\ref{eqn:findisp of g}), (\ref{eqn:finBdisp of g}) and (\ref{eqn:cycle_g}) are satisfied.
To extend  $g_\#^b$ we make use of the following lemma.
\begin{lem}\label{lem:unif acyclic lemma}
If $k$ is as above, then for every $j\in \bbZ$ there  exists a number $R=R_j$ and a function $\mu_j=\mu:\bbR_{\geqslant 0}\rightarrow \bbR_{\geqslant 0}$ such that the following holds.  Suppose $\tau$ is a  $j$-cycle of the form 
\[\tau=\sum_{\sigma\in \Sigma_k}\sigma\otimes \partial\lambda^\sigma+\sum_{i<k,\rho\in \Sigma_i}\rho\otimes \tau^\rho\] 
in either $\bE^k_\bullet$ or $[\bB_\bullet]_k\otimes \bD^b_\bullet$. Let $Y=\supp_X(\tau)\cup\bigcup_{\sigma\in \Sigma_k}p^{p_k(\sigma)}_{j+1-k}(\lambda^\sigma)$ and $Q\coloneqq \diam(Y)$. Then $\sigma$ is the boundary of a chain of the form 
\[\omega=\sum_{\sigma\in \Sigma_k}\sigma\otimes \lambda^\sigma+\sum_{i<k,\rho\in \Sigma_i}\rho\otimes \omega^\rho,\] 
where  $\supp_X(\omega)\subseteq N_{\mu(Q)}(Y)$ and  $\supp_B(\omega)\subseteq N_R(\supp_B(\tau))$.
\end{lem}
\begin{proof}
We prove Lemma \ref{lem:unif acyclic lemma} only when $\tau\in \bE_\bullet^k$, but a similar argument holds if $\tau\in [\bB_\bullet]_k\otimes \bD^b_\bullet$. We proceed inductively on $k$.  The base case $k=0$ is trivial, since then $\partial\sum_{\sigma\in \Sigma_0}\sigma\otimes \lambda^\sigma=\sum_{\sigma\in \Sigma_0}\sigma\otimes \partial\lambda^\sigma$. We assume that Lemma \ref{lem:unif acyclic lemma} holds when $k=j-1$.

We now suppose $k=j$ and that $\tau$ is a cycle of the required form. For each $\sigma\in \Sigma_j$, we write $\partial \sigma=\sum_{\rho\in \Sigma_{j-1}}n^\sigma_\rho \rho$, where $n^\sigma_\rho\in R$. 
Since $\partial\tau=0$, we use the definition of the boundary map for $\bE^{j}_\bullet$ and (\ref{eqn:cycle_g}) to evaluate all the ``$\rho\otimes -$ terms'' in $\partial\tau$ and deduce that 
	\begin{align*}
	\partial(\sum_{\sigma\in \Sigma_j}n^\sigma_\rho  f_\#^{p_j(\sigma),p_{j-1}(\rho)}(\lambda^\sigma) +(-1)^{j-1}\tau_\rho)=0.
	\end{align*} 
Thus for each $\rho\in \Sigma_{j-1}$, there is some some  $\omega^\rho\in D^{p_{j-1}(\rho)}_\bullet$ such that  \[\partial\omega^\rho=\sum_{\sigma\in \Sigma_j} (-1)^{j-1}n^\sigma_\rho  f_\#^{p_j(\sigma),p_{j-1}(\rho)}(\lambda^\sigma) +\tau_\rho.\]

We set $\nu \coloneqq\sum_{\sigma\in \Sigma_j} \sigma\otimes \lambda^\sigma$. Then
\begin{align*}
\partial\nu
&=\sum_{\sigma\in \Sigma_j}\Big((-1)^{j}\sigma\otimes \partial\lambda^\sigma+
g_\#^{p_j(\sigma)}(\partial\sigma\otimes \lambda^\sigma)\Big)\\
&=\sum_{\sigma\in \Sigma_j}\Big((-1)^{j}\sigma\otimes \partial\lambda^\sigma+\sum_{\rho\in \Sigma_{j-1}} n^\sigma_\rho g_\#^{p_j(\sigma)}(\rho\otimes \lambda^\sigma)\Big) 
\end{align*}
for each $\sigma\in \Sigma_j$. 
We can thus write $\tau+(-1)^{j-1}\partial\nu$ as  \begin{align*}
\sum_{\rho\in \Sigma_{j-1}}\rho\otimes \partial\omega^\rho+\sum_{i<j-1,\alpha\in \Sigma_i}\alpha\otimes \gamma^\alpha
\end{align*}
for some $\{\gamma_\alpha\}_{i<j-1,\alpha\in \Sigma_i}$. 
Since $\tau+(-1)^{j-1}\partial\nu$ is  a cycle of the required form,  we can apply  the inductive hypothesis to $\tau+\partial\nu$. 
\end{proof}
We use this lemma to extend $g_\#^b$ from $[\bB_\bullet]_{k-1}\otimes \bD^b_\bullet$ to $[\bB_\bullet]_k\otimes \bD^b_\bullet$ inductively as follows. We assume that $g^b_\#$ has already been defined on $[\bB_\bullet]_k\otimes[ \bD^b_\bullet]_{j-1}$, and let  $\sigma\in \Sigma_k$ and $\lambda\in \Sigma_j^b$. Using (\ref{eqn:cycle_g}), we see that  $g_\#^b(\partial(\sigma\otimes \lambda))=(-1)^kg_\#^b(\sigma\otimes \partial\lambda)+g^b_\#(\partial\sigma\otimes \lambda)$ is a cycle of the form required by Lemma \ref{lem:unif acyclic lemma}. We thus define $g_\#^b(\sigma\otimes\lambda)$ so that $\partial g_\#^b(\sigma\otimes\lambda)=g_\#^b(\partial(\sigma\otimes \lambda))$ and extend linearly. Using Lemma \ref{lem:unif acyclic lemma}, we can ensure that $g_\#^b(\sigma\otimes \lambda)$ will satisfy (\ref{eqn:findisp of g}), (\ref{eqn:finBdisp of g}) and  (\ref{eqn:cycle_g}). This then allows us to define a boundary map $\partial: \bE_{k+1}\rightarrow\bE_{k}$ as above.
Since $\bB_\bullet$ and  $\bD^b_\bullet$ each have finite displacement, it follows from (\ref{eqn:findisp of g}) and the definition of $\partial$ that the chain complex $\bE_\bullet$ has finite displacement.  We thus see that $\bE_\bullet$ is a metric complex over $X$.
   
We now define a chain map $f_\#^b:  \bE_\bullet \rightarrow \bB_\bullet\otimes \bD^b_\bullet $ for every $b\in B$ so that for every $\sigma\otimes \lambda\in T_{i,j}$ with $w=p_i(\sigma)$, we have \begin{equation}\label{eqn:cycle_f}
f^b_\#(\sigma\otimes \lambda)-\sigma\otimes f_\#^{w,b}(\lambda)\in [\bB_\bullet]_{i-1}\otimes \bD^b_\bullet.
\end{equation} 
This is done using Lemma \ref{lem:unif acyclic lemma} in a similar way to how we defined $g_\#^b$. We may assume that $f_\#^b$ satisfies analogues of (\ref{eqn:findisp of g}) and (\ref{eqn:finBdisp of g}). 

In a similar way, we can also define  a chain homotopy  $h_\#^b$ from $f_\#^b g_\#^b$ to the identity, and a chain homotopy $\overline h_\#^b$ from $g_\#^bf_\#^b$ to the identity. This is possible since for all $b,b'\in B$,   $f_{\#}^{b,b'}f_\#^{b',b}$ is chain homotopic to the identity. The  chain homotopies $h_\#^b$ and $\overline h_\#^b$ also satisfy  analogues of (\ref{eqn:findisp of g}) and (\ref{eqn:finBdisp of g}).
Since $g_\#^b$, $f_\#^b$, $h_\#^b$ and $\overline{h}_\#^b$ all satisfy analogues of satisfy  analogues of (\ref{eqn:findisp of g}) and (\ref{eqn:finBdisp of g}), they are proper chain maps and chain homotopies. In particular $\bE_\bullet$ is properly chain homotopic to $\bB_\bullet\otimes \bD^b_\bullet$ for any $b\in B$.

We now show that $\bE_\bullet$ is uniformly acyclic. More specifically, we pick any  $n\geqslant 0$ and show that $\bE_\bullet$ is uniformly $n$-acyclic.  Let $x\in X$, $r\geqslant 0$ and set $b=p(x)$. 
There is an $r_1\geqslant r$ such that 
	\[f_\#^b(\bE_\bullet[N_r(x)]_n)\subseteq \bB_\bullet[N_{r_1}(b)]\otimes \bD^b_\bullet[N_{r_1}(x)].\] 
Using uniform acylicity of both $\bB_\bullet$ and $\bD^b_\bullet$ and the naturality of the K\"unneth formula in Proposition \ref{prop:kunnethchain}, there is an $r_2\geq r_1$ such that the map 
	\[\widetilde H_k(\bB_\bullet[N_{r_1}(b)]\otimes \bD^b_\bullet[N_{r_1}(x)])\rightarrow  \widetilde H_k(\bB_\bullet[N_{r_2}(b)]\otimes \bD^b_\bullet[N_{r_2}(x)]),\] 
induced by inclusion, is zero for $k\leq n$. 
Thus there is an $r_3\geqslant r$,  such that 
	$\widetilde H_k(\bE_\bullet[N_r(x)])\rightarrow \widetilde H_k(\bE_\bullet[N_{r_3}(x)])$, 
induced by $g_\#^bf_\#^b$, is zero for $k\leq n$. 
Since $g_\#^bf_\#^b$ is proper chain homotopic to the identity map, there is some $r_4\geq  r_3$ such that the map $\widetilde H_k(\bE_\bullet[N_r(x)])\rightarrow \widetilde H_k(\bE_\bullet[N_{r_4}(x)])$, induced by inclusion, is zero for $k\leq n$. We can use  (\ref{eqn:findisp of g}), (\ref{eqn:finBdisp of g}), and the analogous identities  for $f_\#^b$ and $h_\#^b$ to show that $r_4$ can be chosen independently of $x$, and so $\bE_\bullet$ is uniformly $n$-acyclic. This completes the proof of Lemma \ref{lem:defchain complex}.

\section{Groups of cohomological codimension one}\label{sec:main result}
We now begin the proof of our main result.  We first recall the analogues of Stallings' end theorem and Dunwoody's accessibility theorem for quotient spaces that were used in \cite{margolis2019almostnormal}.

If $G$ is finitely generated and $H\alnorm G$, it is shown in \cite{connermihalik2014} and \cite{margolis2019almostnormal} that $e(G/H)=\tilde{e}(G,H)$, where $\tilde{e}(G,H)$ is the Kropholler--Roller number of relative ends of the group pair $(G,H)$ \cite{kropholler1989relative}. It thus follows that if $G$ splits over a subgroup commensurable to $H$, then $e(G/H)>1$. The converse follows from \cite{dunwoody1993splitting} and \cite{scott2000splittings}.  
\begin{thm}[\cite{dunwoody1993splitting}, \cite{scott2000splittings}] \label{thm:rel stallings}
Let $G$ be a finitely generated group containing an almost normal subgroup $H\alnorm G$. Then $G$ splits over a subgroup commensurable to $H$ if and only if $e(G/H)>1$.
\end{thm}

There is also a relative analogue of Dunwoody's accessibility theorem. The following is explicitly stated in \cite{margolis2019almostnormal}, although similar applications of Dunwoody accessibility appear in \cite{mosher2003quasi} and \cite{kronmoller08roughcayley}.
\begin{thm}[{\cite[Theorem 3.24]{margolis2019almostnormal}}]\label{thm:dunwoody accessibility}
Let $G$ be an almost finitely presented containing a finitely generated almost normal subgroup $H\alnorm G$. Then  $G$ is the fundamental group of a graph of groups such that:
\begin{enumerate}
\item every edge group is commensurable to $H$;
\item every vertex group is finitely generated and doesn't split over a subgroup commensurable to $H$.
\end{enumerate}
\end{thm}

We also make use of the following lemma for deducing that vertex groups of graphs of groups are of type $FP_\infty(R)$:
\begin{prop}[{\cite[Proposition 2.13]{bieri1981homological}}]\label{prop:finiteness of gogs}
Suppose that $G$ is a group of type $FP_\infty(R)$ and is the fundamental group of a finite graph of groups $\cG$ in which all edge groups are also of type $FP_\infty(R)$. Then all vertex groups of $\cG$ are of type $FP_\infty(R)$.
\end{prop}
In order to apply Theorem \ref{thm:coarse kunneth} -- the coarse K\"unneth formula --  it simplifies matters considerably if we work over a field. This is because for vector spaces,  all $\tor$ terms vanish and the tensor product $V\otimes W$ is non-zero  if and only if $V$ and $W$ are both non-zero.
\begin{lem}\label{lem:higher cohom vanish}
Let $\bbF$ be a field. Suppose $G$ and $H$ are groups of type $FP_\infty(\bbF)$ such that $H\alnorm G$  and  $\Gcd_\bbF(G)\leq \Gcd_\bbF(H)+1$. Then:
\begin{enumerate}
\item \label{item:highercohom}$\coarse^i(G/H;\bbF)=0$ for $i>1$;
\item \label{item:1cohom}$\coarse^1(G/H;\bbF)\neq 0$ if and only if $\Gcd_\bbF(G)= \Gcd_\bbF(H)+1$;
\item \label{item:finindex} $H$ has finite index in $G$ if and only if $\Gcd_\bbF(G)=\Gcd_\bbF(H)$.
\end{enumerate} 
\end{lem}
\begin{proof}
Let $B$ be the quotient space $G/H$, which is coarsely uniformly acyclic over $\bbF$ by Proposition \ref{prop:acyclicity of quotient space}. Since $H$ is of type $FP_\infty(\bbF)$, it is also coarsely uniformly acyclic over $\bbF$ by Proposition \ref{prop:geomcharoffiniteness}.   Let $n=\Gcd_\bbF(H)$, so $\coarse^n(H;\bbF)\neq 0$ by Propositions \ref{prop:coarse cohom vs cohom} and \ref{prop:gorenstein}.  We  now apply Theorem \ref{thm:coarse kunneth}. As $\Gcd_\bbF(G)\leq n+1$, $\coarse^n(H;\bbF)\otimes\coarse^{i-n}(B;\bbF)=0$ for $i>n+1$, and so $\coarse^{i}(B;\bbF)=0$ for $i>1$. This shows (\ref{item:highercohom}). We now deduce  that $H^{n+1}(G,\bbF G)\cong\coarse^{n+1}(G;\bbF)\cong\coarse^n(H;\bbF)\otimes \coarse^1(B;\bbF)$, and (\ref{item:1cohom}) readily follows.

We now show (\ref{item:finindex}).  If $H$ has finite index in $G$, then  $H^k(G,\bbF G)\cong H^k(H,\bbF H)$ as $\bbF$-modules and so $\Gcd_\bbF(G)=\Gcd_\bbF(H)$.  Conversely, suppose $\Gcd_\bbF(G)=\Gcd_\bbF(H)$. Then by (\ref{item:highercohom}) and (\ref{item:1cohom}) we see $\coarse^i(B;\bbF)=0$ for $i>0$, and so $\coarse^{n}(G;\bbF)\cong\coarse^n(H;\bbF)\otimes \coarse^0(B;\bbF)$. Thus $ \coarse^0(B;\bbF)\neq 0$ as $\Gcd_\bbF(G)=n$, and so $B$ is a  bounded metric space by Proposition \ref{prop:cohomatinf}. It is easy to see that $B=G/H$ is bounded if and only if $H$ has finite index in $G$.
\end{proof}

We now prove our main theorem in the case the coefficient ring is a field. 
\begin{thm}\label{thm:main_technical}
Let $\bbF$ be a field. Suppose $G$ and $H$ are groups of type $FP_\infty(\bbF)$ such that $H\alnorm G$,  $G$ is almost finitely presented and $\Gcd_\bbF(G)=\Gcd_\bbF(H)+1$. Then $G$ is the fundamental group of a graph of groups in which all edge and vertex groups are commensurable to $H$.
\end{thm}
\begin{proof}
Let $n=\Gcd_\bbF(H)$. It follows from Lemma \ref{lem:higher cohom vanish} that $\coarse^1(G/H;\bbF)\neq 0$, and so $G/H$ has more than one end by Proposition \ref{prop:cohomatinf}. Theorem \ref{thm:dunwoody accessibility} tells us $G$ can be written as the fundamental group of a graph of groups $\cG$ in which all edge groups are commensurable to $H$ and no vertex group splits over a subgroup commensurable to $H$.

Let $G_v$ be a vertex group of $G$ and let $H_v$ be any incident edge group. We need only demonstrate that $H_v$ has finite index in $G_v$. Since $H_v$ is commensurable to $H$ and $H\alnorm G$, we deduce that $H_v\alnorm G_v$. By Proposition \ref{prop:finiteness of gogs}, $G_v$ is of type $FP_\infty(\bbF)$. By Proposition \ref{prop:gorenstein},   \[n=\Gcd_\bbF(H_v)\leq \Gcd_\bbF(G_v)\leq n+1=\Gcd_\bbF(H)+1=n+1.\] As $G_v$ doesn't split over a subgroup commensurable to $H_v$, Theorem \ref{thm:rel stallings} and Proposition \ref{prop:cohomatinf} ensure that $\coarse^1(G_v/H_v;\bbF)=0$. By Lemma \ref{lem:higher cohom vanish}, $\Gcd_\bbF(G_v)= n$ and so  $H_v$ is a finite index subgroup of $G_v$, and hence $G_v$ is commensurable to $H$.
\end{proof}
\begin{rem}\label{rem:ccdvsgcd}
Let $d_\bbF(G)\coloneqq \sup\{n\mid H^n(G,\bbF G)\neq 0\}$. 
Theorem \ref{thm:main_technical} actually holds whenever $G$ and $H$ are groups of type $FP_\infty(\bbF)$ such that $H\alnorm G$,  $G$ is almost finitely presented,  $d_\bbF(G)=d_\bbF(H)+1<\infty$ and $d_\bbF(G')\leq d_\bbF(G)$ whenever $G'\leq G$.
\end{rem}

Since Gorenstein and virtual cohomological dimension agree whenever the latter is finite, we deduce the following.
\begin{cor}\label{cor:field vfp}
Let $\bbF$ be a field. Suppose $G$ and $H$ are groups of type $VFP(\bbF)$ such that $H\alnorm G$, $G$ is almost finitely presented and $\vcd_\bbF(G)=\vcd_\bbF(H)+1$. Then $G$ is the fundamental group of a graph of groups in which all edge and vertex groups are commensurable to $H$.
\end{cor}

Applying Proposition \ref{prop:type vf} allows us to prove this theorem for virtual  cohomological dimension over any PID. 
\begin{thm}\label{thm:main vfp}
Let $R$ be a PID and let $G$ and $H$ be groups of type $VFP(R)$ such that $H\alnorm G$, $G$ is almost finitely presented and $\vcd_R(G)=\vcd_R(H)+1$. Then $G$ is the fundamental group of a graph of groups in which all edge and vertex groups are commensurable to $H$.
\end{thm}
We note that when $R=\bbZ$ or $\bbZ_2$, any group of type $VFP(R)$  is necessarily almost finitely presented. 
\begin{proof}
Let $n=\vcd_R(H)$. By Proposition \ref{prop:type vf}, there is a field $R_i$ such that $\vcd_{R_i}(H)=n$. It also follows from Proposition \ref{prop:type vf} that $\vcd_{R_i}(G)\leq \vcd_R(G)$.  As $\vcd_R(G)>\vcd_R(H)$,  $H$ cannot be a finite index subgroup of $G$, so   Lemma \ref{lem:higher cohom vanish} ensures $\vcd_{R_i}(G)=n+1$. We can now apply Corollary \ref{cor:field vfp}.
\end{proof}

If $H$ is a duality group of dimension $n$, then $H^i(H,P)=0$ when $i\neq n$ and $P$ is a projective $\bbZ H$-module. This follows from the characterisation of duality groups in Proposition \ref{prop:duality def}, Proposition 5.2 of \cite{brown1982cohomology}, and the fact that every projective module is the direct summand of a free module. 
Combining this observation with  Proposition 2.6 of \cite{kropholler2006spectral} and Theorem \ref{thm:main vfp}, we deduce the following:
\begin{thm}
Let $G$ be a finitely generated group of cohomological dimension $n+1$ and let $H\alnorm G$ be a duality group of dimension $n$. Then $G$ splits as a graph of groups in which every vertex and edge group is commensurable to $H$.
\end{thm}

\section{Almost normal subgroups of one-relator groups}\label{sec:1rel}
In this section we generalise Corollary 4 of \cite{bieri1978cd2}.

\onerel

To prove this, we we use the Euler characteristic of a group as defined in \cite[\S IX]{brown1982cohomology}. This is defined for all groups of type $VFP$ and satisfies the following property: $\chi(G)= \frac{\chi(G')}{[G:G']}$ whenever $G$ is of type $VFP$ and $[G:G']<\infty$. We can use \cite[Proposition IX.7.3, e and e']{brown1982cohomology} to deduce that  \begin{align}\chi(A*_CB)=\chi(A)+\chi(B)-\chi(C)\label{eqn:amalg}\end{align} when $A*_CB$, $A$, $B$ and $C$ all have type $VFP$, and \begin{align}\label{eqn:hnn}\chi(A*_C)=\chi(A)-\chi(C)\end{align} when  $A*_C$, $A$, $B$ and $C$ all have type $VFP$.
We use these observations to calculate the Euler characteristic of certain graphs of groups.
\begin{lem}\label{lem:gogeuler}
Suppose $H\leq G$ are groups of type $VFP$ such that  $G$ is the fundamental group of a non-trivial reduced graph of groups $\cG$ in which every vertex and edge group contains $H$ as a subgroup of finite index. Then:
\begin{enumerate}
\item If $\chi(H)=0$, then $\chi(G)=0$.
\item If $\chi(H)\neq 0$, then $\frac{\chi(G)}{\chi(H)}\leq 0$. Moreover, $\chi(G)=0$ if and only if  the Bass-Serre tree of $\cG$ is a line, i.e. either $G=A*_CB$ where $[A:C]=[B:C]=2$, or $G=A*_\phi$ where $\phi:A\rightarrow A$ is an isomorphism.
\end{enumerate}
\end{lem}
\begin{proof}
Since $\chi(G')=\chi(G)[G:G']$ when $G'$ is a finite index subgroup of $G$, we may assume without loss of generality that $G$ has type $FP$.
If $\chi(H)=0$, then every edge and vertex  group of $\cG$ has zero Euler characteristic, so  (\ref{eqn:amalg}) and (\ref{eqn:hnn}) ensure $\chi(G)=0$. We thus assume $\chi(H)\neq 0$.

We proceed by induction on the number of edges of $\cG$. We assume $\cG$ has one edge. If $G=A*_CB$, then we deduce from (\ref{eqn:amalg}) that \[\chi(G)=\frac{\chi(H)}{[C:H]}\Bigg(\frac{1}{[A:C]}+\frac{1}{[B:C]}-1\Bigg).\] Since $\cG$ is reduced, $[A:C],[B:C]\geq 2$ and so the ratio $\frac{\chi(G)}{\chi(H)}$ is non-positive. If $G=A*_C$, then \[\chi(G)=\frac{\chi(H)}{[C:H]}\Bigg(\frac{1}{[A:C]}-1\Bigg),\] and so $\frac{\chi(G)}{\chi(H)}$ is also non-positive.
For the inductive step, the same argument shows that if  $\frac{\chi(G)}{\chi(H)}\leq 0$ and $B$ and $C$ both contain $H$ as a finite index subgroup, then both $\frac{\chi(G*_C)}{\chi(H)}$ and $\frac{\chi(G*_CB)}{\chi(H)}$ are negative as required. 

Since $\chi(G)$ is negative when $\cG$ has more than one edge, $\chi(G)=0$ only if $\cG$ has one edge. If $G=A*_CB$, then the above formula tells us that $[A:C],[B:C]=2$ if and only if $\chi(G)=0$. Similarly, if $G=A*_C$, then $\chi(G)=0$ if and only if $C=A$. 
\end{proof}

We use this to deduce the following:
\begin{prop}\label{prop:eulerchar}
Suppose $G$ and $H$ are groups of type $VFP$ such that $\vcd(G)=\vcd(H)+1$, $\chi(G)\leq 0$ and $\chi(H)\leq 0$. Then $\chi(G)=0$ and one of the following holds:
\begin{enumerate}
\item $\chi(H)=0$;
\item there is a subgroup $N\vartriangleleft G$, commensurable to $H$, such that $G/N\cong \bbZ$ or $\bbZ_2*\bbZ_2$.
\end{enumerate}
\end{prop}
\begin{proof}
By Theorem \ref{thm:main_technical}, $G$ splits as a graph of groups in which all vertex and edge groups are commensurable to $H$. By replacing $H$ with a finite index subgroup, we may assume $H$ is contained in every vertex and edge group; this clearly doesn't affect the hypothesis that $\chi(H)\leq 0$.
If $\chi(G)\neq 0$, then it follows from Lemma \ref{lem:gogeuler} that $\chi(H)\neq 0$ and $\frac{\chi(G)}{\chi(H)}<0$, which cannot be the case as we assumed $\chi(G)\leq 0$ and $\chi(H)\leq 0$. Thus $\chi(G)=0$, and the conclusion follows from Lemma \ref{lem:gogeuler}.
\end{proof}

\begin{proof}[Proof of Theorem \ref{thm:1rel}]
We adapt the proof of \cite[Corollary 4]{bieri1978cd2}. We first show that $G$ is finitely generated. If not, then we can write $G=G'*F$, where $F$ is free and $G'$ is a finitely generated 1-relator group containing $H$. If $F\neq 1$, then $H$ cannot be an almost normal subgroup of $G$, so we deduce that $G$  is finitely generated and so has a presentation of the form $\langle g_1,\dots g_n\mid r^m\rangle$, where $n\geq 1$ and $m>0$.  As was noted in \cite[\S VIII.11 Example 3]{brown1982cohomology}, $G$ is of type $VFP$ and has virtual cohomological dimension at most 2. Moreover, $\chi(G)=1-n+\frac{1}{m}\leq 0$ as in \cite{bieri1978cd2}.

Since $H$ is an infinite index subgroup of $G$, Proposition \ref{prop:type vf} and Lemma \ref{lem:higher cohom vanish} ensure that $\vcd(H)<\vcd(G)$. As $H$ is infinite, $\vcd(H)> 0$, so $\vcd(H)=1$ and $\vcd(G)=2$. By Theorem \ref{thm:main vfp}, $G$ splits as a graph of groups all of whose edge and vertex groups are commensurable to $H$. The Euler characteristic of a free group of rank $r$ is $1-r$, so  $\chi(H)\leq 0$ by Theorem \ref{thm:stallingsCD1}. By Proposition \ref{prop:eulerchar}, $\chi(G)=0$ and so $m=1$ and $n=1$. Thus $G$ is torsion free and 2-generated. Proposition \ref{prop:eulerchar} also ensures that either $\chi(H)=0$, in which case either $H$ is infinite cyclic or  $H$ is commensurable to a normal free subgroup $N$ with  $G/N\cong \bbZ$ or $\bbZ_2*\bbZ_2$.
\end{proof}

\section{Almost normal subgroups of virtual duality groups}\label{sec:duality groups}
We examine almost normal subgroups  of  virtual duality groups defined in \cite{bierieckmann1973duality}. The following characterisation of duality groups can be taken as a definition for the purposes of this article:
\begin{prop}[{\cite[Proposition 3.1]{bieri1976normal}}]\label{prop:duality def}
Let $R$ be a commutative ring. We say a group $G$ is a virtual duality group of dimension $n$ over $R$ if and only if the following hold:
\begin{enumerate}
\item $G$ is type $VFP(R)$;
\item $H^i(G,R G)=0$ for $i\neq n$;
\item $H^n(G,RG)$ is a flat $R$-module.
\end{enumerate}
Moreover, $G$ is said to be a Poincar\'e duality group of dimension $n$ over $R$ if  $H^n(G,R G)\cong R$ as $R$-modules.
\end{prop}
In this section we prove the following:
\begin{thm}\label{thm:duality}
Let $R$ be a PID, and let $H\alnorm G$ be groups of type $VFP(R)$.
\begin{enumerate}
\item \label{item:duality 1} If $G$ is a virtual duality group over $R$, then so is $H$.
\item \label{item:duality 2} If $H$ is a virtual duality group over $R$ and $\vcd_R(G)=\vcd_R(H)+1$, then $G$ is a virtual duality group over $R$.
\end{enumerate}
\end{thm}
\begin{proof}
(\ref{item:duality 1}): We recall that $R_0$ is the field of fractions of $R$, and that for each prime $p\in \bbP_R$, $R_p$ is the field $R/pR$. It follows from Proposition \ref{prop:unif0-acycl} that $G$ is a duality group of dimension $n$ over the field $R_i$ for every $i\in \bbP_R\cup\{0\}$. 
We now apply Theorem \ref{thm:coarse kunneth} to calculate the coarse cohomology of $G$ in terms of the coarse cohomology of $H$ and $G/H$. In particular, $\coarse ^{m_i}(H;R_i)\neq 0$  for some unique $m_i\in \bbZ$.   Thus $H$ is an $m_i$-dimensional virtual duality group over the field $R_i$. 
Bieri showed that $H$ is a virtual duality group of dimension $m$ over $R$ if and only if it is a virtual duality group of dimension $m$ over $R_i$ for every $i\in \bbP_R\cup\{0\}$ \cite[Theorem 3.7]{bieri1976normal}. Thus to show $H$ is a virtual duality group, we need only show that $m_0=m_i$ for every $i\in \bbP_R$. 

Since $H$ is a duality group of dimension $m\coloneqq m_0$ over $R_0$, $\coarse^i(H;R_0)=0$ for $i\neq m$. It follows from Proposition \ref{prop:univ coefficient} that $\coarse^i(H,R)$ is a torsion $R$-module for $i\neq m$. 
In particular, $\coarse^i(H;R)\otimes _R\coarse ^j(G/H;R)$ is torsion for every $i\neq m$. Moreover, for any $R$-modules $A$ and $B$, $\tor_1^R(A,B)$ is always torsion  whenever $R$ is an integral domain \cite[Theorem 7.15]{rotman2009homalg}.
As $\coarse^n(G;R)$ is torsion-free, Theorem \ref{thm:coarse kunneth} implies \[\coarse^n(G;R)\cong \coarse^m(H;R) \otimes_R \coarse^{n-m}(G/H;R).\]

Thus for any $p\in \bbP_R$, Proposition \ref{prop:univ coefficient} ensures
\begin{align*}
\coarse^n(G,R_p)&\cong \coarse^n(G,R)\otimes_R R_p\bigoplus \tor_1^R(\coarse^{n+1}(G;R),R_p)\\
&\cong \Big(\coarse^m(H;R) \otimes_R \coarse^{n-m}(G/H;R)\Big)\otimes_R R_p\\
&\cong \Big(\coarse^m(H;R)\otimes_{R} R_p\Big) \otimes_{R_p} \Big(\coarse^{n-m}(G/H;R)\otimes_{R} R_p\Big).
\end{align*}
Hence $\coarse^m(H;R)\otimes_{R} R_p\neq 0$, so by Proposition \ref{prop:univ coefficient},   $\coarse^m(H;R_p)\neq 0$ for each $p\in \bbP_R$. Since $H$ is a  duality group of dimension $m_p$ over $R_p$, we  deduce that $m_p=m$ for every $p\in \bbP_R$ as required.

(\ref{item:duality 2}): This follows easily from Theorem \ref{thm:main vfp} and  \cite[\S 9.7]{bieri1981homological}.
\end{proof}

Specialising to the case of Poincar\'e duality groups we have the following:
\begin{thm}\label{thm:pdn_gps}
Suppose $R$ is a PID and $G$ is a virtual Poincar\'e duality group over $R$. If $H\alnorm G$ is of type $VFP(R)$, then $H$ is a virtual Poincar\'e duality group over $R$. Moreover, if $\vcd_R(G)=\vcd_R(H)+1$, then $H$ is a commensurable to a subgroup $N\vartriangleleft G$ such that $G/N\cong \bbZ$ or $\bbZ_2*\bbZ_2$.
\end{thm}
\begin{proof}
Suppose $G$ is a virtual Poincar\'e duality group over $R$ of dimension $n$. As in the proof of Theorem \ref{thm:duality}, 
\[R\cong \coarse^{n}(G;R)\cong \coarse^m(H;R) \otimes_R \coarse^{n-m}(G/H;R),\] and $\coarse^m(H;R)$ is a  non-zero torsion-free $R$-module. Pick some non-zero $x\in  \coarse^{n-m}(G/H;R)$. As $\coarse^m(H;R)$ is torsion-free, hence flat, we have an injection of $R$-modules \[\coarse^m(H;R)\cong \coarse^m(H;R)\otimes_R Rx \hookrightarrow \coarse^{n}(G;R) \cong R.\] Thus $\coarse^m(H;R)$ is isomorphic to an ideal of $R$, hence isomorphic to $R$.

When $\vcd_R(G)=\vcd_R(H)+1$, the preceding argument shows that $\coarse^{1}(G/H;R)\cong R$, and so $e(G/H)=\tilde e(G,H)=2$. The result follows from \cite[Corollary 3.4]{kropholler1989relative}. Alternatively, one can combine Theorem \ref{thm:main vfp} and  \cite[Corollary 3.20]{margolis2019almostnormal} to deduce that $G$ acts on a line with vertex and edge stabilizers commensurable to $H$.
\end{proof}
We use this to classify finitely generated almost normal subgroups of 3-manifold groups.
In the following corollary, a \emph{surface group} is the fundamental group of a closed orientable surface other than the 2-sphere, and a \emph{3-manifold group} is the fundamental group of an irreducible orientable closed  3-manifold.
\begin{cor}
Let $G$ be a 3-manifold group. If $H\alnorm G$ is finitely generated, infinite and of infinite index, then $H$ is commensurable to a normal subgroup $N\vartriangleleft G$, where $N$ is either infinite cyclic or a surface group.
\end{cor}
\begin{proof}
It follows from the sphere theorem that $G$ is the fundamental group of an aspherical manifold, hence $G$ is  a Poincar\'e duality group of dimension 3. Combining results of Scott and Strebel, $H$ must be of type $FP$ \cite{scott1973coherence, strebel1977remark} and $m\coloneqq \cd(H)< \cd(G)$. It follows from Theorem \ref{thm:pdn_gps} and a result of Eckmann--M\"uller \cite{eckmannmuller1982plane} that when $\cd(H)=2$, $H\vartriangleleft G$ and $H$  a surface group.

If $\cd(H)=1$, then it follows from Theorem \ref{thm:stallingsCD1} and Theorem \ref{thm:pdn_gps} that $H=\langle h\rangle\cong \bbZ$. Let $S=\{s_1,\dots, s_t\}$ be a generating set of $G$. Since $H\alnorm G$, for every $s_i\in S$, there exist integers $n_i,m_i\in \bbZ$ such that $s_ih^{n_i}s_i^{-1}=h^{m_i}$. By a theorem of Kropholler, $n_i=\pm m_i$ for each $i$ \cite{kropholler1990centrality}. Thus $\langle h^N\rangle\vartriangleleft G$, where $N=\lcm(n_1, n_2, \dots, n_t)$.
\end{proof}

\section*{Acknowledgements}
The author would like to thank Peter Kropholler and Michah Sageev for comments and suggestions relating to this work.

\bibliography{bibliography/bibtex} 
\bibliographystyle{amsalpha}
\end{document}